\documentclass[a4paper]{amsart}
\usepackage[english]{babel}
\usepackage{amsmath,amsthm,amssymb,amsfonts}

\let\Oldepsilon\epsilon
\let\Oldvarepsilon\varepsilon
  \renewcommand{\varepsilon}{\Oldepsilon}
  \renewcommand{\epsilon}{\Oldvarepsilon}

\DeclareMathOperator{\Kern}{\mathcal{K}}
\DeclareMathOperator{\supp}{\mathrm{supp}}

\newcommand{\N}{\mathbb{N}}
\newcommand{\Z}{\mathbb{Z}}
\newcommand{\R}{\mathbb{R}}
\newcommand{\C}{\mathbb{C}}

\newcommand{\Space}{\mathrm{X}}

\newcommand{\dist}{\varrho}
\newcommand{\dil}{D}

\newcommand{\tc}{\,:\,}
\newcommand{\defeq}{\mathrel{:=}}

\newcommand{\leftopenint}{\left]}
\newcommand{\rightopenint}{\right[}
\newcommand{\leftclosedint}{\left[}
\newcommand{\rightclosedint}{\right]}

\DeclareMathOperator*{\esssup}{\mathrm{ess\,sup}}

\newtheorem{thm}{Theorem}
\newtheorem{prp}[thm]{Proposition}
\newtheorem{cor}[thm]{Corollary}
\newtheorem{lem}[thm]{Lemma}

\newcommand{\vecL}{\mathbf{L}}
\newcommand{\vecT}{\mathbf{T}}

\newcommand{\vecONE}{\tilde{1}}

\newcommand{\done}{{d_1}}
\newcommand{\dtwo}{{d_2}}
\newcommand{\jone}{{j}}
\newcommand{\jtwo}{{k}}
\newcommand{\alphaone}{{\alpha}}
\newcommand{\alphatwo}{{\beta}}

\newcommand{\chr}{\chi}

\newcommand{\lbN}{\gamma}
\newcommand{\tlbN}{\tilde \gamma}
\newcommand{\bdL}{k}

\begin{document}
\title[Multipliers for Grushin operators]{A sharp multiplier theorem for Grushin operators in arbitrary dimensions}
\author{Alessio Martini}
\address{Alessio Martini \\ Mathematisches Seminar \\ Christian-Albrechts-Universit\"at zu Kiel \\ Ludewig-Meyn-Str.\ 4 \\ D-24118 Kiel \\ Germany}
\email{martini@math.uni-kiel.de}
\author{Detlef M\"uller}
\address{Detlef M\"uller \\ Mathematisches Seminar \\ Christian-Albrechts-Universit\"at zu Kiel \\ Ludewig-Meyn-Str.\ 4 \\ D-24118 Kiel \\ Germany}
\email{mueller@math.uni-kiel.de}

\thanks{The first-named author gratefully acknowledges the support of the Alexander von Humboldt Foundation.}

\begin{abstract}
In a recent work by A.\ Martini and A.\ Sikora, sharp $L^p$ spectral multiplier theorems for the Grushin operators acting on $\R^{\done}_{x'} \times \R^{\dtwo}_{x''}$ and defined by the formula
\[
L=-\sum_{\jone=1}^{\done}\partial_{x'_\jone}^2 - \left(\sum_{\jone=1}^{d_1}|x'_\jone|^2\right) 
\sum_{\jtwo=1}^{\dtwo}\partial_{x''_\jtwo}^2
\]
are obtained in the case $\done \geq \dtwo$. Here we complete the picture by proving sharp results in the case $\done < \dtwo$. Our approach exploits $L^2$ weighted estimates with ``extra weights'' depending only on the second factor of $\R^{\done} \times \R^{\dtwo}$ (in contrast with the mentioned work, where the ``extra weights'' depend on the first factor) and gives a new unified proof of the sharp results without restrictions on the dimensions.
\end{abstract}

\maketitle

\section{Introduction}

Let $\Space$ be $\R^{\done} \times \R^{\dtwo}$ with Lebesgue measure, and let $L$ be the Grushin operator on $\Space$, that is,
\[L = -\Delta_{x'} - |x'|^2 \Delta_{x''},\]
where $x',x''$ denote the two components of a point $x \in \R^\done \times \R^\dtwo$, while $\Delta_{x'},\Delta_{x''}$ are the corresponding partial Laplacians, and $|x'|$ is the Euclidean norm of $x'$. Since $L$ is an essentially self-adjoint operator on $L^2(\Space)$, a functional calculus for $L$ can be defined via spectral integration and, for all Borel functions $F : \R \to \C$, the operator $F(L)$ is bounded on $L^2(\Space)$ if and only if the function $F$, which is called \emph{spectral multiplier}, is essentially bounded with respect to the spectral measure.

The aim of this work is to give sufficient conditions for the $L^p$-boundedness (for $p \neq 2$) of an operator of the form $F(L)$, in terms of smoothness properties of the multiplier $F$. Namely, let $W_2^s(\R)$ denote the $L^2$ Sobolev space on $\R$ of (fractional) order $s$, and define a ``scale-invariant local Sobolev norm'' by the formula
\[ \|F\|_{MW_2^s} = \sup_{t > 0} \|\eta \, F_{(t)} \|_{W_2^s},\]
where $F_{(t)}(\lambda) = F(t\lambda)$, and $\eta \in C^\infty_c(\leftopenint 0,\infty \rightopenint)$ is a nontrivial auxiliary function (different choices of $\eta$ give rise to equivalent local norms). Our main results then read as follows.

\begin{thm}\label{thm:mihlinhoermander}
Suppose that a function $F : \R \to \C$ satisfies
\[\|F\|_{MW_2^s} < \infty\]
for some $s > (\done+\dtwo)/2$. Then the operator $F(L)$ is of weak type $(1,1)$ and bounded on $L^p(\Space)$ for all $p \in \leftopenint 1,\infty \rightopenint$. In addition, for all $p \in \leftopenint 1,\infty \rightopenint$,
\[
\|F(L)\|_{L^1 \to L^{1,\infty}} \leq C_s \|F\|_{MW_2^s}, \qquad \|F(L)\|_{L^p \to L^{p}} \leq C_{p,s} \|F\|_{MW_2^s}.
\]
\end{thm}

\begin{thm}\label{thm:bochnerriesz}
Suppose that $\kappa > (\done+\dtwo-1)/2$ and $p \in \leftclosedint 1,\infty \rightclosedint$. Then the Bochner-Riesz means $(1-tL)_+^\kappa$ are bounded on $L^p(\Space)$ uniformly in $t \in \leftclosedint 0,\infty \rightopenint$.
\end{thm}

These results are sharp, in the sense that the lower bounds on the order of differentiability $s$ in Theorem \ref{thm:mihlinhoermander} and on the order $\kappa$ of the Bochner-Riesz means in Theorem \ref{thm:bochnerriesz} cannot be decreased.

In the case $\done \geq \dtwo$, the results above are contained in a joint work of the first-named author and Adam Sikora \cite{martini_grushin}, to which we refer for a discussion of the related literature (see also \cite{folland_hardy_1982,mauceri_vectorvalued_1990,christ_multipliers_1991,hebisch_multiplier_1993,mller_spectral_1994,hebisch_functional_1995,mller_marcinkiewicz_1996,cowling_spectral_2001,duong_plancherel-type_2002,meyer_estimates_2006,robinson_analysis_2008,jotsaroop_riesz_2011}), and for a proof of the mentioned sharpness (based on \cite{kenig_divergence_1982}). In fact, \cite{martini_grushin} contains some results for the case $\done < \dtwo$ too, which however are not sharp. The new approach presented here differs from the one of \cite{martini_grushin} even in the case $\done \geq \dtwo$, and gives a unified treatment of the sharp results without any restriction on the pair $(\done,\dtwo)$.

\section{Structure of the proof}

Let $\dist$ be the control distance on $\Space$ associated to the Grushin operator $L$, and denote by $B(x,r)$ the open $\dist$-ball of center $x$ and radius $r$, and by $|B(x,r)|$ its Lebesgue measure. Denote moreover by $\Kern_{F(L)}$ the integral kernel of the operator $F(L)$. As shown in \cite{martini_grushin}, Theorems~\ref{thm:mihlinhoermander} and \ref{thm:bochnerriesz} are consequences of the following $L^1$ weighted estimate (corresponding to \cite[Corollary 14]{martini_grushin} in the case $\done \leq \dtwo$).

\begin{prp}\label{prp:l1estimate}
For all $R > 0$, $\alpha \geq 0$, $\beta > \alpha + (\done+\dtwo)/2$, and for all functions $F : \R \to \C$ such that $\supp F \subseteq \leftclosedint R^2, 4R^2 \rightclosedint$,
\begin{equation}\label{eq:l1estimate}
\esssup_{y \in \Space} \|(1+R\dist(\cdot,y))^\alpha \, \Kern_{F(L)}(\cdot,y)\|_1 \leq C_{\alpha,\beta} \|F_{(R^2)}\|_{W_2^\beta}.
\end{equation}
\end{prp}

This estimate in turn follows via H\"older's inequality from an $L^2$ weighted estimate of the form
\begin{equation}\label{eq:plancherel}
\esssup_{y \in \Space} |B(y,1/R)|^{1/2} \|w_R(x,y)^\gamma \, (1+R\dist(\cdot,y))^\alpha \, \Kern_{F(L)}(\cdot,y)\|_2 \leq C_{\alpha,\beta,\gamma} \|F_{(R^2)}\|_{W_2^\beta}
\end{equation}
for suitable weight functions $w_R : \Space \times \Space \to \leftclosedint 0,\infty \rightopenint$ and constraints on $\alpha,\beta,\gamma \in \leftclosedint 0,\infty \rightopenint$.

In \cite{martini_grushin} the weights $w_R(x,y)$ depend only on the first components $x',y'$ of $x,y$, and the proof of \eqref{eq:plancherel} is based on a subelliptic estimate satisfied by $L$. Such approach corresponds to the one adopted in \cite{hebisch_multiplier_1993} for the sublaplacian on a Heisenberg(-type) group $G$, where a weight function is used, that depends only on (the projection of the variable on) the first layer of $G$.

On the other hand, other works in the setting of Heisenberg groups \cite{mller_spectral_1994,mller_marcinkiewicz_1996} exploit weight functions depending on both layers.

The approach presented below differs from all the previous ones, since we use weight functions $w_R$ depending only on the second components $x'',y''$ of the variables $x,y$. In place of the subelliptic estimate used in \cite{martini_grushin}, here we perform a careful analysis based on the properties of the Hermite functions; in this sense, we are closer to the spirit of \cite{mller_spectral_1994,mller_marcinkiewicz_1996}, where instead identities for Laguerre functions are exploited.

We remark that the $L^2$ estimate \eqref{eq:plancherel} without the weights $w_R$ (that is, when $\gamma = 0$) holds true if $\beta > \alpha$, and this implies the $L^1$ estimate \eqref{eq:l1estimate} when $\beta > \alpha + Q/2$, where $Q$ is the ``homogeneous dimension'' $d_1 + 2d_2$ of the doubling space $\Space$ with distance $\dist$ and Lebesgue measure \cite{duong_singular_1999,robinson_analysis_2008}. The purpose of the ``extra weights'' $w_R$ is to pass from the homogeneous dimension $Q$ to the topological dimension $\done+\dtwo$. Since these two quantities differ by the dimension $\dtwo$ of the second factor of $\R^\done \times \R^\dtwo$, it appears necessary, when $d_2$ is larger than $d_1$, to employ weights $w_R(x,y)$ depending not only on the first components $x',y'$. In fact the technique presented here, in contrast to the one in \cite{martini_grushin}, does not put any constraint on the dimensions.

\section{Weighted estimates and discrete differentiation}
Given a point $x = (x',x'') \in \Space$, we denote by $x'_\jone$ and $x''_\jtwo$ the $\jone$-th component of $x'$ and the $\jtwo$-th component of $x''$. For all $\jone \in \{1,\dots,\done\}$, $\jtwo \in \{1,\dots,\dtwo\}$, let then $L_{\jone}$ and $T_{\jtwo}$ be the differential operators on $\Space$ given by
\[L_{\jone} = (-i\partial_{x'_\jone})^2 + (x'_\jone)^2 \sum_{l=1}^{\dtwo} (-i\partial_{x''_l})^2, \qquad\qquad T_{\jtwo} = -i\partial_{x''_{\jtwo}}.\]
If $(\dil_r)_{r > 0}$ is the family of dilations on $\Space$ defined by	
\[\dil_r(x',x'') = (rx',r^2 x''),\]
then
\[
L_\jone(f \circ \dil_r) = r^2 (L_\jone f) \circ \dil_r, \qquad T_{\jtwo} (f \circ \dil_r) = r^2 (T_{\jtwo} f) \circ \dil_r.
\]
The Grushin operator $L$ on $\Space$ is the sum $L_1 + \dots + L_\done$.

As shown in \cite{martini_grushin}, the operators $L_1,\dots,L_\done,T_1,\dots,T_{\dtwo}$ have a joint functional calculus; moreover, if $\vecL$ and $\vecT$ denote the ``vectors of operators'' $(L_1,\dots,L_\done)$ and $(T_1,\dots,T_\dtwo)$, one can obtain a quite explicit formula for the integral kernel $\Kern_{G(\vecL,\vecT)}$ of an operator $G(\vecL,\vecT)$ in the functional calculus, in terms of the Hermite functions. Namely, for all $\ell \in \N$, let $h_\ell$ denote the $\ell$-th Hermite function, that is,
\[
h_\ell(t) = (-1)^\ell (\ell! \, 2^\ell \sqrt{\pi})^{-1/2} e^{t^2/2} \left(\frac{d}{dt}\right)^\ell e^{-t^2},
\]
and set, for all $n \in \N^{\done}$, $u \in \R^{\done}$, $\xi \in \R^{\dtwo}$,
\[\tilde h_n(u,\xi) = |\xi|^{\done/4} h_{n_1}(|\xi|^{1/2} u_1) \cdots h_{n_{\done}}(|\xi|^{1/2} u_{\done}).\]
Finally, denote by $e_1,\dots,e_{\done}$ the standard basis vectors of $\R^\done$, and by $\vecONE$ the element $(1,\dots,1) = e_1 + \dots + e_\done$ of $\N^{\done}$.

\begin{prp}
For all bounded Borel functions $G : \R^\done \times \R^{\dtwo} \to \C$ compactly supported in $\R^\done \times (\R^{\dtwo} \setminus \{0\})$, if
\begin{equation}\label{eq:reparametrization}
m(n,\xi) = \begin{cases}
G(|\xi|(2n+\vecONE),\xi) &\text{when $n \in \N^\done$,}\\
0 &\text{when $n \in \Z^\done \setminus \N^\done$,}
\end{cases}
\end{equation}
then
\begin{equation}\label{eq:kernel}
\Kern_{G(\vecL,\vecT)}(x,y) 
= (2\pi)^{-\dtwo} \int_{\R^\dtwo} \sum_{n \in \N^\done} m(n,\xi) \, \tilde h_n(y',\xi) \, \tilde h_n(x',\xi) \, e^{i \langle \xi , x''-y'' \rangle} \,d\xi
\end{equation}
for almost all $x,y \in \Space$.
\end{prp}
\begin{proof}
See \cite[Proposition 5]{martini_grushin}.
\end{proof}

The relation \eqref{eq:kernel} between the kernel $\Kern_{G(\vecL,\vecT)}$ and the multiplier $G$ -- or rather its reparametrization $m$ -- involves a partial Fourier transform. This suggests that applying a suitable multiplication operator to the kernel may correspond to applying a differential operator to the multiplier. The presence of the Hermite expansion, however, make things more complicated, and leads one to considering discrete difference operators as well as continuous derivatives on the spectral side. In order to give a precise form to these observations, we introduce a certain amount of notation.

For all $\ell \in \Z$, set $a_\ell = \sqrt{\ell(\ell-1)}$ if $\ell > 0$ and $a_\ell = 0$ otherwise. Let us define the following operators on functions $f : \Z^\done \times \R^\dtwo \to \C$:
\begin{align*}
\tau_\jone f(n,\xi) &= f(n+2e_\jone, \xi),\\
\delta_\jone f(n,\xi) &= f(n,\xi) - f(n-2e_\jone,\xi),\\
N_{\jone,\rho,s} f(n,\xi) &= \begin{cases} a_{n_\jone+2\rho} f(n,\xi) &\text{if $s = 0$,}\\
N_{\jone,\rho,s-1} f(n,\xi) - N_{\jone,\rho-1,s-1} f(n,\xi) &\text{if $s > 0$,}\end{cases}\\
\partial_\jtwo f(n,\xi) &= \frac{\partial}{\partial \xi_\jtwo} f(n,\xi),
\end{align*}
for all $\jone \in \{1,\dots,\done\}$, $\jtwo \in \{1,\dots,\dtwo\}$, $\rho \in \Z$, $s \in \N$. Note that $\tau_j$ is invertible, and $\delta_\jone f = f - \tau_\jone^{-1} f$. We will also use the multiindex notation as follows:
\[\tau^\alphaone = \tau_1^{\alphaone_1} \cdots \tau_\done^{\alphaone_\done}, \qquad \delta^\alphaone = \delta_1^{\alphaone_1} \cdots \delta_\done^{\alphaone_\done}, \qquad \partial^\alphatwo = \partial_1^{\alphatwo_1} \cdots \partial_\dtwo^{\alphatwo_\dtwo},\]
for all $\alphaone \in \N^\done$ and $\alphatwo \in \N^\dtwo$; in fact, $\tau^\alphaone$ is defined for all $\alpha \in \Z^\done$. Inequalities between multiindices, such as $\alpha \leq \alpha'$, shall be understood componentwise. Moreover $|\cdot|_1$ will denote the $1$-norm, that is, for all $t \in \R^d$, $|t|_1 = |t_1| + \dots + |t_d|$.

For convenience, set $h_\ell = 0$ for all $\ell < 0$, and extend the definition of $\tilde h_n$ to all $n \in \Z^\done$; hence $\tilde h_n = 0$ for all $n \in \Z^\done \setminus \N^\done$.

\begin{prp}\label{prp:derivativekernelformula}
Let $G : \R^\done \times \R^{\dtwo} \to \C$ be smooth and compactly supported in $\R^\done \times (\R^{\dtwo} \setminus \{0\})$, and let $m(n,\xi)$ be defined by \eqref{eq:reparametrization}.
For all $\alphatwo \in \N^\dtwo$, we have
\begin{multline*}
(x''-y'')^\alphatwo \, \Kern_{G(\vecL,\vecT)}(x,y) \\
= \int_{\R^\dtwo} \sum_{n \in \Z^\done} \Bigl[ \sum_{\iota \in I_\alphatwo} \Theta_\iota(\xi) \, \partial^{\alphatwo^{\iota}} \mathcal{N}_\iota \, \tau^{\tilde\alphaone^{\iota}} \delta^{\alphaone^{\iota}} m(n,\xi) \, \tilde h_{n+2r^\iota}(y',\xi) \Bigr] \, \tilde h_n(x',\xi) \, e^{i \langle \xi , x''-y'' \rangle} \,d\xi
\end{multline*}
for almost all $x,y \in \Space$, where $I_\alphatwo$ is a finite set and, for all $\iota \in I_\alphatwo$,
\begin{enumerate}
\item[(i)] $\alphatwo^{\iota} \in \N^\dtwo$ and $\alphatwo^{\iota} \leq \alphatwo$,
\item[(ii)] $\alphaone^{\iota},\tilde\alphaone^{\iota} \in \N^\done$ and $|\alphaone^\iota|_1 + |\alphatwo^\iota|_1 \leq |\alphatwo|_1$,
\item[(iii)] if $|\alphatwo|_1> 0$ then $|\alphaone^\iota|_1 +|\alphatwo^\iota|_1 > 0$,
\item[(iv)] $r^\iota \in \Z^\done$ and $|r^\iota|_1 \leq |\alphatwo|_1$,
\item[(v)] $\Theta_\iota$ is a smooth function on $\R^\dtwo \setminus \{0\}$, homogeneous of degree $|\alphatwo^\iota|_1 - |\alphatwo|_1$,
\item[(vi)] $\mathcal{N}_\iota$ is a composition product of the form
\begin{equation}\label{eq:compositionproduct}
N_{1,\rho^1_1,s^1_{1}} \cdots N_{1,\rho^1_{u_1},s^1_{u_1}} \cdots N_{\done,\rho^\done_1,s^\done_1} \cdots N_{\done,\rho^\done_{u_\done},s^\done_{u_\done}}
\end{equation}
with $u_1+\dots+u_\done \leq |\alphatwo|_1-|\alphatwo^\iota|_1$ and
\begin{gather*}
s^\jone_1 + \dots + s^\jone_{u_\jone} = u_\jone - \alphaone^\iota_\jone, \qquad s^\jone_l - |\alphatwo|_1 \leq \rho^\jone_l \leq |\alphatwo|_1, \\ \max\{0,1-\rho^\jone_1,\dots,1-\rho^\jone_{u_\jone}\} \geq \alphaone^\iota_\jone-\tilde\alphaone^\iota_\jone
\end{gather*}
for all $\jone \in \{1,\dots,\done\}$ and $l \in \{1,\dots,u_\jone\}$.	
\end{enumerate}
\end{prp}
\begin{proof}
Because of \eqref{eq:kernel}, we are reduced to proving that
\begin{equation}\label{eq:hermiteexpansionderivative}
\begin{split}
\left(\frac{\partial}{\partial \xi}\right)^\alphatwo \sum_{n \in \Z^\done} & m(n,\xi) \, \tilde h_{n}(y',\xi) \, \tilde h_n(x',\xi) \\
&= \sum_{\iota \in I_\alphatwo} \sum_{n \in \Z^\done}  \Theta_\iota(\xi) \, \partial^{\alphatwo^{\iota}} \mathcal{N}_\iota \, \tau^{\tilde\alphaone^{\iota}} \delta^{\alphaone^{\iota}} m(n,\xi) \, \tilde h_{n+2r^\iota}(y',\xi) \, \tilde h_n(x',\xi)
\end{split}
\end{equation}
where $I_\alphatwo$, $\alphatwo^\iota$, $\alphaone^\iota$, $\tilde\alphaone^\iota$, $r^\iota$, $\Theta_\iota$, $\mathcal{N}_\iota$ are as in the statement above.

This formula can be proved by induction on $|\alphatwo|_1$. For $|\alphatwo|_1 = 0$ it is trivially verified. For the inductive step, from well-known properties of the Hermite functions \cite[p.\ 2]{thangavelu_lectures_1993} we deduce
\[2 t h'_\ell(t) = a_\ell h_{\ell-2}(t) - a_{\ell+2} h_{\ell+2}(t) - h_\ell(t)\]
for all $\ell \in \Z$ and $t \in \R$. Correspondingly, for all $n,r \in \Z^\done$, $x',y' \in \R^\done$ and $\xi \in \R^\dtwo \setminus \{0\}$,
\[\begin{split}
\frac{\partial}{\partial \xi_\jtwo} \left[ \tilde h_{n+2r}(y',\xi) \, \tilde h_n(x',\xi) \right] 
= \frac{\xi_\jtwo}{4|\xi|^2} \sum_{\jone=1}^{\done} &\Bigl[ a_{n_\jone+2r_\jone}\tilde h_{n+2(r-e_\jone)}(y',\xi) \, \tilde h_n(x',\xi) \\
&- a_{n_\jone+2(r_\jone+1)}\tilde h_{n+2(r+e_\jone)}(y',\xi) \, \tilde h_n(x',\xi) \\
&+ a_{n_\jone} \tilde h_{n+2r}(y',\xi) \, \tilde h_{n-2e_\jone}(x',\xi) \\
&- a_{n_\jone+2}\tilde h_{n+2r}(y',\xi) \, \tilde h_{n+2e_\jone}(x',\xi) \Bigr].
\end{split}\]
Hence, for all smooth $f : \Z^\done \times \R^\dtwo \to \C$ compactly supported in $\Z^\done \times (\R^\dtwo \setminus \{0\})$,
\begin{equation}\label{eq:hermiteexpansionfirstderivative}
\begin{split}
\frac{\partial}{\partial \xi_\jtwo} \sum_{n \in \Z^\done} &f(n,\xi) \, \tilde h_{n+2r}(y',\xi) \, \tilde h_n(x',\xi)
= \sum_{n \in \Z^\done} \Bigl[ \partial_\jtwo f(n,\xi) \, \tilde h_{n+2r}(y',\xi) \\
&+ \frac{\xi_\jtwo}{4|\xi|^2} \sum_{\jone =1}^\done N_{\jone,1,0} \tau_\jone \delta_\jone f(n,\xi) \, \tilde h_{n+2(r+e_\jone)}(y',\xi) \\
&- \frac{\xi_\jtwo}{4|\xi|^2} \sum_{\jone =1}^\done \varepsilon_{r_\jone} \sum_{\rho = 1-(r_\jone)_-}^{(r_\jone)_+} N_{\jone,\rho+1,1} f(n,\xi) \, \tilde h_{n+2(r+e_\jone)}(y',\xi) \\
&+ \frac{\xi_\jtwo}{4|\xi|^2} \sum_{\jone =1}^\done \varepsilon_{r_\jone} \sum_{\rho = 1-(r_\jone)_-}^{(r_\jone)_+} N_{\jone,\rho,1} f(n,\xi) \, \tilde h_{n+2(r-e_\jone)}(y',\xi) \\
&+ \frac{\xi_\jtwo}{4|\xi|^2} \sum_{\jone =1}^\done N_{\jone,0,0} \delta_\jone f(n,\xi) \, \tilde h_{n+2(r-e_\jone)}(y',\xi) \Bigr] \, \tilde h_n(x',\xi),
\end{split}
\end{equation}
where, for all $\ell \in \Z$,
\[
\varepsilon_\ell = \begin{cases}
+1 &\text{if $\ell \geq 0$,}\\
-1 &\text{if $\ell < 0$}
\end{cases}
\qquad\text{and}\qquad
(\ell)_\pm = \max\{\pm \ell, 0\}.
\]
By taking the derivative $\partial/\partial \xi_k$ of both sides of \eqref{eq:hermiteexpansionderivative}, applying \eqref{eq:hermiteexpansionfirstderivative} to each summand in the right-hand side, and exploiting the ``commutation relations''
\[
\tau_\jone N_{l,\rho,s} = \begin{cases}
N_{l,\rho+1,s} \tau_\jone &\text{if $\jone = l$,}\\
N_{l,\rho,s} \tau_\jone &\text{if $\jone \neq l$,}
\end{cases}
\qquad
\delta_\jone N_{l,\rho,s} = \begin{cases}
N_{l,\rho,s+1} + N_{l,\rho-1,s} \delta_\jone, &\text{if $\jone = l$,}\\
N_{l,\rho,s} \delta_\jone, &\text{if $\jone \neq l$,}\\
\end{cases}
\]
one obtains the analogue of \eqref{eq:hermiteexpansionderivative} where $\beta$ is increased by $1$ in the $k$-th component.
\end{proof}

Plancherel's formula, together with the orthonormality of the Hermite functions and the finiteness of the index set $I_\alphatwo$, then yields the following estimate.

\begin{cor}\label{cor:derivativekernelformula}
Under the hypotheses of Proposition~\ref{prp:derivativekernelformula}, for all $\beta \in \N^\dtwo$ and almost all $y \in \Space$,
\begin{multline}\label{eq:l2derivativekernelformula}
\int_\Space \left| (x''-y'')^\alphatwo \, \Kern_{G(\vecL,\vecT)}(x,y)\right|^2 \,dx \\
\leq C_{\alphatwo} \int_{\R^\dtwo} \sum_{n \in \N^\done} \sum_{\iota \in I_\alphatwo} |\xi|^{2|\alphatwo^\iota|_1 - 2|\alphatwo|_1} \, \left| \mathcal{N}_\iota \tau^{\tilde\alphaone^\iota} \delta^{\alphaone^\iota} \partial^{\alphatwo^\iota} m(n,\xi) \right|^2 \, \tilde h_{n+2r^\iota}^2(y',\xi) \,d\xi.
\end{multline}
\end{cor}

\section{From discrete to continuous}

The next few lemmata will be of use in clarifying the meaning of the various terms appearing in the right-hand side of \eqref{eq:l2derivativekernelformula}.

Note that for all $\xi \in \R^\dtwo$, $\tau_j f(\cdot,\xi)$, $\delta_j f(\cdot,\xi)$, $N_{j,\rho,s} f(\cdot,\xi)$ depend only on $f(\cdot,\xi)$. In other words, the operators $\tau_j$, $\delta_j$, $N_{j,\rho,s}$ and their compositions can be considered as operators on functions $\Z^\done \to \C$.

\begin{lem}\label{lem:discretecontinuous}
Let $f : \Z^\done \to \C$ have a smooth extension $\tilde f : \R^\done \to \C$, and let $\alphaone \in \N^\done$, $\tilde\alphaone \in \Z^\done$; then
\[\tau^{\tilde\alphaone} \delta^\alphaone f(n) = 2^{|\alphaone|_1} \int_{J_{\alphaone,\tilde\alphaone}} \partial^\alphaone \tilde f(n-s) \,d\nu_{\alphaone,\tilde\alphaone}(s)\]
for all $n \in \Z^\done$, where $J_{\alphaone,\tilde\alphaone} = \prod_{\jone=1}^\done \leftclosedint -2\tilde\alphaone_j, 2\alphaone_j  - 2\tilde\alphaone_j \rightclosedint$ and $\nu_{\alphaone,\tilde\alphaone}$ is a Borel probability measure on $J_{\alphaone,\tilde\alphaone}$. In particular
\[|\tau^{\tilde\alphaone} \delta^\alphaone f(n)|^2 \leq 2^{2|\alphaone|_1} \int_{J_{\alphaone,\tilde\alphaone}} |\partial^\alphaone \tilde f(n-s)|^2 \,d\nu_{\alphaone,\tilde\alphaone}(s)\]
and
\[|\tau^{\tilde\alphaone} \delta^\alphaone f(n)| \leq 2^{|\alphaone|_1} \sup_{s \in J_{\alphaone,\tilde\alphaone}} |\partial^\alphaone \tilde f(n-s)|\]
for all $n \in \Z^\done$.
\end{lem}
\begin{proof}
Iterated application of the fundamental theorem of integral calculus gives
\[\delta^{\alphaone} f(n) = 2^{|\alphaone|_1} \int_{\leftclosedint 0,1\rightclosedint^{\alphaone_1}} \cdots \int_{\leftclosedint 0,1\rightclosedint^{\alphaone_\done}} \partial^{\alphaone} \tilde f(n_1-2|s_1|_1,\dots,n_\done-2|s_\done|_1) \,ds_1 \,\dots \,ds_\done\]
and the conclusion follows by taking as $\nu_{\alphaone,\tilde\alphaone}$ the push-forward of the uniform distribution on $\prod_{\jone=1}^\done \leftclosedint 0,1\rightclosedint^{\alphaone_\jone}$ via the map $(s_1,\dots,s_\done) \mapsto (2|s_1|_1-2\tilde\alphaone_1,\dots,2|s_\done|_1-2\tilde\alphaone_\done)$, and by H\"older's inequality.
\end{proof}

\begin{lem}\label{lem:coefficientestimate}
Let $\mathcal{N}$ be the product \eqref{eq:compositionproduct}, and let $f : \Z^\done \to \C$. Then
\begin{enumerate}
\item $\mathcal{N} f(n) = 0$ for all $n \in \Z^\done$ such that $n_j < 2 \max\{-\infty,1-\rho^\jone_1,\dots,1-\rho^\jone_{u_\jone}\}$ for at least one $\jone \in \{1,\dots,\done\}$, and
\item $|\mathcal{N} f(n)| \leq C_{\mathcal{N}} |f(n)| \prod_{\jone=1}^\done (2|n_\jone|+1)^{u_\jone-(s^\jone_1+\dots+s^\jone_{u_\jone})}$ for all $n \in \Z^\done$.
\end{enumerate}
\end{lem}
\begin{proof}
It is sufficient to prove the conclusion in the case where the product $\mathcal{N}$ is made of a single factor $N_{\jone,\rho,s}$.

$N_{\jone,\rho,s}$ is a multiplication operator, with multiplier $\tau_\jone^{\rho} \delta_\jone^{s} w_\jone$, where $w_\jone(n) = a_{n_\jone}$. Since $a_\ell = 0$ when $\ell < 2$, inductively we obtain $\tau_\jone^{\rho} \delta_\jone^{s} w_\jone(n) = \delta_\jone^{s} w(n+2\rho e_j) = 0$ when $n_\jone < 2(1-\rho)$, and part (1) follows.

The function $w_\jone : \Z^\done \to \C$ can be extended to a smooth function $\tilde w_\jone : \R^\done \to \C$ such that $\tilde w_\jone(t) = \sqrt{t_\jone(t_\jone-1)}$ if $t_\jone > 3/2$, say, and $\tilde w(t) = 0$ if $t_\jone \leq 1$. By Leibniz' rule, if $t_\jone > 3/2$, then
\[\partial_\jone^s \tilde w(t) = \sum_{v = 0}^s c_{s,v} \, t_\jone^{1/2-v} (t_\jone-1)^{1/2-(s-v)}\]
for some constants $c_{s,v} \in \R$, and in particular $|\partial_\jone^s \tilde w(t)| \leq C_s t_\jone^{1-s}$ if $t_\jone > 3/2$. Lemma~\ref{lem:discretecontinuous} then gives that
\[|\tau_\jone^{\rho} \delta_\jone^{s} w_\jone(n)| \leq C_{s} \sup_{2\rho - 2s \leq \theta \leq 2\rho} (n_\jone + \theta)^{1-s} \leq C_{\rho,s} (2|n_\jone|+1)^{1-s}\]
for all $n$ with $n_j \geq 2(1-\rho+s)$. Possibly by increasing the constant, the inequality $|\tau_\jone^{\rho} \delta_\jone^{s} w(n)| \leq C_{\rho,s} (2|n_\jone|+1)^{1-s}$ extends to all $n \in \Z^\done$, and part (2) follows.
\end{proof}

For all $d \in \N \setminus \{0\}$, $\ell \in \N$, $u \in \R^d$, set
\[H_{d,\ell}(u) = \sum_{\substack{n \in \N^d \\ |n|_1 = \ell}} h_{n_1}^2(u_1) \cdots h_{n_d}^2(u_d).\]
For the reader's convenience, we rewrite here the known bounds for the functions $H_{d,\ell}$ that will be used in the following (see \cite[Lemma~8]{martini_grushin} and references therein).
\begin{lem}
Let $d \in \N \setminus \{0\}$ and set $[\ell] = 2\ell + d$.
If $d = 1$ then, for all $\ell \in \N$,
\begin{equation}\label{eq:muckenhoupt}
H_{1,\ell}(u) \leq \begin{cases}
C([\ell]^{1/3} + |u^2-[\ell]|)^{-1/2} &\text{for all $u \in \R$,}\\
C\exp(-cu^2) &\text{when $u^2 \geq 2[\ell]$.}
\end{cases}
\end{equation}
If $d \geq 2$ then, for all $\ell \in \N$,
\begin{equation}\label{eq:higherbounds}
H_{d,\ell}(u) \leq \begin{cases}
C_d [\ell]^{d/2-1} &\text{for all $u \in \R^d$,}\\
C_d \exp(-c_d |u|_\infty^2) &\text{when $|u|_\infty^2 \geq 2[\ell]$,}
\end{cases}
\end{equation}
where $|u|_\infty = \max\{|u_1|,\dots,|u_d|\}$.
\end{lem}

The following lemma is a refined version of \cite[Lemma 9]{martini_grushin}.

\begin{lem}\label{lem:newhermiteestimate}
Let $d \in \N \setminus \{0\}$ and set $[\ell] = 2\ell + d$. Let $(b_\ell)_{\ell \in \N}$ be a sequence in $\leftopenint 0,\infty \rightopenint$ such that, for some $\kappa \in \leftclosedint 1,\infty \rightopenint$,
\[\kappa^{-1} \leq b_\ell / [\ell] \leq \kappa\]
for all $\ell \in \N$. In the case $d=1$, suppose further that
\[|b_\ell - [\ell]| \leq \kappa [\ell]^{2/3}\]
for all $\ell \in \N$. Then, for all $x \in \leftopenint 0,\infty \rightopenint$ and $u \in \R^d$,
\begin{equation}\label{eq:hermite}
\sum_{\substack{\ell \in \N \\ [\ell] \leq x}} H_{d,\ell}(b_\ell^{-1/2} u) \leq C_{d,\kappa} 
\begin{cases}
x^{d/2} &\text{in any case,}\\
\exp(-|u|^2/(c_{d,\kappa}\, x)) &\text{if $|u| \geq c_{d,\kappa}\, x$,}
\end{cases}
\end{equation}
for some $c_{d,\kappa} \in \leftclosedint 1,\infty \rightopenint$.
\end{lem}
\begin{proof}
We may assume that $x \geq 1$, otherwise the left-hand side of \eqref{eq:hermite} vanishes.

In order to exploit the bounds \eqref{eq:muckenhoupt} and \eqref{eq:higherbounds}, we consider several cases.

First of all, in the case $|u|_\infty \geq x \sqrt{2 \kappa}$, if $[\ell] \leq x$, then $b_\ell \leq \kappa x$, hence
\[|b_\ell^{-1/2} u|_\infty^2 \geq  |u|_\infty^2 / (\kappa x) \geq 2 x \geq 2[\ell],\]
and therefore
\begin{equation}\label{eq:partialhermite}
\begin{split}
\sum_{[\ell] \leq x} H_{d,\ell}(b_\ell^{-1/2} u) &\leq C_d x \exp(- c_d |u|_\infty^2 / (\kappa x)) \\
&\leq C_d \exp(- c_d |u|_\infty^2 / (2\kappa x)) \, \sup_{t \geq 1} (t \exp(-c_d t)).
\end{split}
\end{equation}
Thus the second inequality in \eqref{eq:hermite} is proved (by a suitable choice of $c_{d,\kappa}$).

In the case $d > 1$, the first inequality in \eqref{eq:hermite} is immediately proved because
\[\sum_{[\ell] \leq x} H_{d,\ell}(b_\ell^{-1/2} u) \leq C_d \sum_{[\ell] \leq x} [\ell]^{d/2-1} \leq C_d \, x^{d/2}.\]
In the case $d = 1$, instead, we need to split the sum in \eqref{eq:hermite} in several parts:
\[
\sum_{\substack{[\ell] \leq x}} H_{1,\ell}(b_\ell^{-1/2} u) = \sum_{\substack{[\ell] \leq x \\ [\ell] \leq |u|/\sqrt{2\kappa}}} + \sum_{\substack{[\ell] \leq x \\ |u|/\sqrt{2\kappa} < [\ell] < |u| \sqrt{2\kappa}}} + \sum_{\substack{[\ell] \leq x \\ [\ell] \geq |u| \sqrt{2\kappa}}}.
\]

The first and the last part are the easiest to control. In fact, the part where $[\ell] \leq |u|/\sqrt{2\kappa}$ is controlled by a constant because of \eqref{eq:partialhermite}. Moreover, in the part where $|u| \sqrt{2\kappa} \leq [\ell] \leq x$, we have $u^2/b_\ell \leq [\ell]/2$, hence
\[\sum_{|u| \sqrt{2\kappa} \leq [\ell] \leq x} H_{1,\ell}(b_\ell^{-1/2} u) \leq C \sum_{[\ell] \leq x} [\ell]^{-1/2} \leq C \, x^{1/2}.\]

The middle part instead requires a further splitting:
\[
\sum_{\substack{[\ell] \leq x \\ |u|/\sqrt{2\kappa} < [\ell] < |u| \sqrt{2\kappa}}} = \sum_{\substack{[\ell] \leq x \\ |u|/\sqrt{2\kappa} < [\ell] \\ [\ell] \leq |u|-\kappa [\ell]^{2/3}}} + \sum_{\substack{[\ell] \leq x \\ |u|/\sqrt{2\kappa} < [\ell] < |u| \sqrt{2\kappa} \\ |u|-\kappa [\ell]^{2/3} < [\ell] < |u|+\kappa [\ell]^{2/3}}} + \sum_{\substack{[\ell] \leq x \\ |u|+\kappa [\ell]^{2/3} \leq [\ell] \\ [\ell] < |u|\sqrt{2\kappa}}}.
\]

In the part where $|u|/\sqrt{2\kappa} < [\ell] \leq |u|-\kappa [\ell]^{2/3}$, we have $|u| \geq 1+\kappa$ and
\[[\ell] \leq |u| - 1, \qquad b_\ell \leq |u|, \qquad 1/\sqrt{2\kappa} \leq [\ell]/|u| < 1,\]
hence
\[\left|\frac{u^2}{b_\ell} - [\ell]\right| \geq |u| \left(1-\frac{[\ell]}{|u|}\right),\]
so this part of the sum is majorized by
\[
C_{\kappa} \, \frac{x^{1/2}}{|u|} \sum_{|u|/\sqrt{2\kappa} < [\ell] \leq |u|-\kappa [\ell]^{2/3}} \left(1-\frac{[\ell]}{|u|}\right)^{-1/2} 
\leq C_{\kappa} \, x^{1/2} \int_{1/\sqrt{2\kappa}}^1 (1-t)^{-1/2} \,dt,
\]
and the last integral is finite.

In the part where $|u|+\kappa [\ell]^{2/3} \leq [\ell] < |u|\sqrt{2\kappa}$, we have $|u| \geq 1/\sqrt{2\kappa}$ and
\[[\ell] \geq |u| + 1, \qquad b_\ell \geq |u|, \qquad 1 < [\ell]/|u| \leq \sqrt{2\kappa},\]
hence
\[\left|\frac{u^2}{b_\ell} - [\ell]\right| \geq |u| \left(\frac{[\ell]}{|u|}-1\right),\]
so this part of the sum is majorized by
\[
C \frac{x^{1/2}}{|u|} \sum_{|u|+\kappa [\ell]^{2/3} \leq [\ell] < |u|\sqrt{2\kappa}} \left(\frac{[\ell]}{|u|}-1\right)^{-1/2} 
\leq C \, x^{1/2} \int_{1}^{\sqrt{2\kappa}} (t-1)^{-1/2} \,dt,
\]
and the last integral is finite.

In the part where $|u|/\sqrt{2\kappa} < [\ell] < |u| \sqrt{2\kappa}$ and $|u|-\kappa [\ell]^{2/3} < [\ell] < |u|+\kappa [\ell]^{2/3}$ there are at most $\kappa (2\kappa)^{1/3} |u|^{2/3}$ summands, and moreover $|u| \leq x \sqrt{2\kappa}$, hence this part of the sum is majorized by
\[C_\kappa |u|^{2/3} |u|^{-1/6} \leq C_\kappa \, x^{1/2},\]
and we are done.
\end{proof}

We may now give a more explicit form to the right-hand side of \eqref{eq:l2derivativekernelformula}, in terms of a Sobolev norm of the multiplier, in the case we restrict to the functional calculus for the Grushin operator $L$ alone. In order to avoid divergent series, however, it is convenient at first to truncate the multiplier along the spectrum of $\vecT$.

\begin{lem}\label{lem:weightedplancherel}
Let $\chi \in C^\infty_c(\leftopenint 0,\infty \rightopenint)$ be such that $\supp \chi \subseteq [1/2,2]$.
Let $F : \R \to \C$ be smooth and such that $\supp f \subseteq K$ for some compact $K \subseteq \leftopenint 0,\infty\rightopenint$. For all $r \in \leftclosedint0,\infty\rightopenint$ and $M \in \leftclosedint 1,\infty \rightopenint$, if $F_M : \R \times \R^\dtwo \to \C$ is defined by
\[F_M(\lambda,\xi) = F(\lambda) \, \chi(\lambda/(M |\xi|)),\]
then
\begin{multline*}
\int_\Space \left| |x''-y''|^r \, \Kern_{F_M(L,\vecT)}(x,y)\right|^2 \,dx \\
\leq C_{\chi,K,r} M^{2r-\dtwo} \left( \chr_{\leftclosedint0,c_{K,r} \rightclosedint}(|y'|/M) + e^{-|y'|} \right) \|F\|^2_{W_2^r}
\end{multline*}
for almost all $y \in \Space$.
\end{lem}
\begin{proof}
Without loss of generality, we can restrict to the case $r \in \N$, the remaining values of $r$ being recovered by interpolation. It is then sufficient to prove
\begin{multline*}
\int_\Space \left| (x''-y'')^\alphatwo \, \Kern_{F_M(L,\vecT)}(x,y)\right|^2 \,dx \\
\leq C_{\chi,K,\alphatwo} M^{2|\alphatwo|_1-\dtwo} \left( \chr_{\leftclosedint0,c_{K,\beta} \rightclosedint}(|y'|/M) + e^{-|y'|} \right) \|F\|^2_{W_2^{|\alphatwo|_1}}
\end{multline*}
for all $\alphatwo \in \N^\dtwo$ and almost all $y \in \Space$.

Set $\langle t \rangle = |2t+\vecONE|_1 = 2|t|_1 + \done$ for all $t \in \R^\done$. An estimate for the left-hand side of the previous inequality is given by Corollary~\ref{cor:derivativekernelformula}, by taking $m(n,\xi) = F(|\xi| \langle n \rangle) \, \chi(\langle n \rangle / M)$ for $n \in \N^\done$ and $m(n,\xi) = 0$ for $n \in \Z^\done \setminus \N^\done$. This estimate, combined with Lemma~\ref{lem:coefficientestimate}, gives
\begin{multline*}
\int_\Space \left| (x''-y'')^\alphatwo \, \Kern_{F_M(L,\vecT)}(x,y)\right|^2 \,dx 
\leq C_{\alphatwo}  \sum_{\iota \in I_\alphatwo} \int_{\R^\dtwo} \sum_{n \geq \lbN^\iota} |\xi|^{2 |\alphatwo^\iota|_1-2 |\alphatwo|_1} \\
\times (2n_1+1)^{2\alphaone^\iota_1} \cdots (2n_\done+1)^{2\alphaone^\iota_\done} \left| \tau^{\tilde\alphaone^\iota} \delta^{\alphaone^\iota} \partial^{\alphatwo^\iota} m(n,\xi) \right|^2 \, \tilde h^2_{n+2r^\iota}(y',\xi) \,d\xi,
\end{multline*}
where $\lbN^\iota \defeq (\lbN^\iota_1,\dots,\lbN^\iota_\done)$ and $\lbN^\iota_\jone \defeq 2\max\{0,1-\rho^\jone_1,\dots,1-\rho^\jone_{u_\jone}\} \geq 2(\alphaone^\iota_\jone-\tilde\alphaone^\iota_\jone)$ for all $\jone \in \{1,\dots,\done\}$. If $\tilde m$ is a smooth extension of $m$, then Lemma~\ref{lem:discretecontinuous} gives
\begin{multline*}
\int_\Space \left| (x''-y'')^\alphatwo \, \Kern_{F_M(L,\vecT)}(x,y)\right|^2 \,dx 
\leq C_{\alphatwo}  \sum_{\iota \in I_\alphatwo} \int_{J_\iota} \int_{\R^\dtwo} \sum_{n \geq \tlbN^\iota} |\xi|^{2 |\alphatwo^\iota|_1-2 |\alphatwo|_1} \\
\times \langle n \rangle^{2|\alphaone^\iota|_1} \left| \partial_t^{\alphaone^\iota} \partial_{\xi}^{\alphatwo^\iota} \tilde m(n-s,\xi) \right|^2 \, \tilde h^2_{n}(y',\xi) \,d\xi \,d\nu_\iota(s),
\end{multline*}
where $\tlbN^\iota \defeq (\tlbN^\iota_1,\dots,\tlbN^\iota_\done)$, $\tlbN^\iota_\jone \defeq \max\{0,\lbN^\iota_\jone+2r^\iota_\jone\} \geq 2(r^\iota_\jone-\tilde\alphaone^\iota_\jone+\alphaone^\iota_\jone)$ for all $\jone \in \{1,\dots,\done\}$, $J_\iota = \prod_{\jone=1}^\done \leftclosedint 2 (r^\iota_\jone - \tilde\alphaone^\iota_\jone), 2 (r^\iota_\jone - \tilde\alphaone^\iota_\jone + \alphaone^\iota_\jone) \rightclosedint$, and $\nu_\iota$ is a probability measure on $J_\iota$. Note that all the components of the first argument $n-s$ of $\tilde m$ in the right-hand side of the previous inequality are always nonnegative, since $n \geq \tlbN^\iota$ and $s \in J_\iota$.

A smooth extension $\tilde m$ of $m$ is given by
\[\tilde m(t,\xi) = F(|\xi|(2t_1+\dots+2t_\done+\done)) \, \chi((2t_1+\dots+2t_\done+\done)/M)\]
for $\xi \in \R^\dtwo \setminus \{0\}$ and $t \in \leftopenint -1/2,\infty \rightopenint^{d_1}$. An inductive argument then shows that
\[\partial_\xi^{\alphatwo^\iota} \partial_t^{\alphaone^\iota} \tilde m(t,\xi) = \sum_{\substack {0 \leq a \leq |\alphaone^\iota|_1 \\ 0 \leq b \leq |\alphatwo^\iota|_1}} M^{a-|\alphaone^\iota|_1} \chi^{(|\alphaone^\iota|_1 - a)}(\langle t \rangle/M) \, \Psi_{\alphatwo^\iota,a,b}(\xi) \, \langle t \rangle^{b} F^{(a+b)}(|\xi| \langle t \rangle)\]
for all $t \in \leftclosedint 0,\infty \rightopenint^\done$, where the $\Psi_{\alphatwo^\iota,a,b} : \R^d \setminus \{0\} \to \C$ are smooth functions, homogeneous of degree $a+b - |\alphatwo^\iota|_1$. Hence
\[|\partial_\xi^{\alphatwo^\iota} \partial_t^{\alphaone^\iota} \tilde m(t,\xi)|^2 \leq C_{\chi,\iota} \sum_{v = 0}^{|\alphaone^\iota|_1 +|\alphatwo^\iota|_1} |\xi|^{2v -2|\alphatwo^\iota|_1} \, M^{2v-2|\alphaone^\iota|_1} |F^{(v)}(|\xi| \langle t \rangle)|^2 \, \tilde\chi(\langle t \rangle/M)\]
for all $t \in \leftclosedint 0,\infty \rightopenint^\done$, where $\tilde\chi$ is the characteristic function of $\leftclosedint 1/2,2\rightclosedint$; therefore, since $|\alphaone^\iota|_1 +|\alphatwo^\iota|_1 \leq |\alphatwo|_1$ for all $\iota \in I_\beta$, we have
\begin{multline*}
\int_\Space \left| (x''-y'')^\alphatwo \, \Kern_{F_M(L,\vecT)}(x,y)\right|^2 \,dx 
\leq C_{\chi,\alphatwo} \sum_{v = 0}^{|\alphatwo|_1} M^{2v} \\
\times \sum_{\iota \in I_\alphatwo} \sum_{\substack{n \geq \tlbN^\iota \\ c_\iota^{-1} \leq \langle n \rangle/M \leq c_\iota}} \int_{J_\iota} \int_{\R^\dtwo} |\xi|^{2 v-2 |\alphatwo|_1}
 \, | F^{(v)}(|\xi| \langle n-s \rangle) |^2 \, \tilde h^2_{n}(y',\xi) \,d\xi \,d\nu_\iota(s),
\end{multline*}
where $c_\iota \in \leftclosedint 2,\infty \rightclosedint$ is chosen so that $2/c_\iota \leq \langle n \rangle / \langle n-s \rangle \leq c_\iota/2$ for all $n \geq \tlbN^\iota$ and $s \in J_\iota$.

If $\bdL_\iota \defeq \tlbN^\iota_1 + \dots + \tlbN^\iota_\done$, $\tilde J_\iota$ is the interval in $\R$ which is the image of $J_\iota$ via the map $(s_1,\dots,s_\done) \mapsto s_1 + \dots + s_d$, and $\tilde \nu_\iota$ is the corresponding push-forward of $\nu_\iota$ on $\tilde J_\iota$, then $\bdL_\iota \geq \max \tilde J_\iota$ and
\begin{multline*}
\int_\Space \left| (x''-y'')^\alphatwo \, \Kern_{F_M(L,\vecT)}(x,y)\right|^2 \,dx 
\leq C_{\chi,\alphatwo} \sum_{v = 0}^{|\alphatwo|_1} \sum_{\iota \in I_\alphatwo} \sum_{\substack{\ell \geq \bdL_\iota \\ c_\iota^{-1} \leq [\ell]/M \leq c_\iota}} M^{2v} \\
\times \int_{\tilde J_\iota} \int_{\R^\dtwo} |\xi|^{2 v-2 |\alphatwo|_1} 
| F^{(v)}(|\xi| [\ell-s] ) |^2 \, \sum_{n \tc |n|_1 = \ell} \tilde h^2_{n}(y',\xi) \,d\xi \,d\tilde\nu_\iota(s),
\end{multline*}
where $[\ell] = 2\ell + \done$. Note that $\sum_{n \tc |n|_1 = \ell} \tilde h^2_{n}(y',\xi) = |\xi|^{\done/2} H_{\done,\ell}(|\xi|^{1/2} y')$, and that the integrand in $\xi \in \R^\dtwo$ depends only on $|\xi|$, hence
\begin{multline*}
\int_\Space \left| (x''-y'')^\alphatwo \, \Kern_{F_M(L,\vecT)}(x,y)\right|^2 \,dx 
\leq C_{\chi,\alphatwo} \sum_{v = 0}^{|\alphatwo|_1} \sum_{\iota \in I_\alphatwo} \sum_{\substack{\ell \geq \bdL_\iota \\ c_\iota^{-1} \leq [\ell]/M \leq c_\iota}} M^{2v} \\
\times \int_{\tilde J_\iota} \int_0^\infty \lambda^{2 v-2 |\alphatwo|_1+d_1/2+d_2} 
| F^{(v)}(\lambda [\ell-s] ) |^2 \, H_{\done,\ell}(\lambda^{1/2} y') \,\frac{d\lambda}{\lambda} \,d\tilde\nu_\iota(s).
\end{multline*}
Note that $[\ell-s] \sim [\ell] \sim M$ in the domain of summation/integration in the right-hand side; a rescaling in the integral in $\lambda$, together with the fact that $\supp F \subseteq K$ and $K \subseteq \leftopenint 0,\infty \rightopenint$ is compact then gives
\begin{multline*}
\int_\Space \left| (x''-y'')^\alphatwo \, \Kern_{F_M(L,\vecT)}(x,y)\right|^2 \,dx 
\leq C_{\chi,K,\alphatwo} \sum_{v = 0}^{|\alphatwo|_1} \sum_{\iota \in I_\alphatwo} \int_0^\infty | F^{(v)}(\lambda) |^2 \\
\times  M^{2|\alphatwo|_1-\dtwo-\done/2} \int_{\tilde J_\iota} \sum_{\substack{\ell \geq \bdL_\iota \\ c_\iota^{-1} \leq [\ell]/M \leq c_\iota}} H_{\done,\ell} \left(\frac{\lambda^{1/2} y'}{[\ell-s]^{1/2}} \right) \,d\tilde\nu_\iota(s) \,d\lambda.
\end{multline*}
On the other hand, from Lemma~\ref{lem:newhermiteestimate} we easily obtain
\[M^{-\done/2}  \sum_{\substack{\ell \geq \bdL_\iota \\ c_\iota^{-1} \leq [\ell]/M \leq c_\iota}} H_{\done,\ell} \left(\frac{\lambda^{1/2} y'}{[\ell-s]^{1/2}} \right) \leq C_{K,\beta} \left(\chr_{\leftclosedint 0, c_{K,\beta}\rightclosedint}(|y'|/M) + e^{-|y'|}\right),\]
uniformly in $\iota \in I_\beta$, $s \in \tilde J_\iota$, $\lambda \in K$, by choosing $c_{K,\beta}$ sufficiently large, and we are done.
\end{proof}

Define the weight $w : \Space \times \Space \to \leftclosedint 1,\infty \rightopenint$ by
\[w(x,y) = 1+\frac{|x''-y''|}{1+|y'|}.\]

\begin{prp}\label{prp:weightedplancherel}
Let $F : \R \to \C$ be smooth and such that $\supp F \subseteq K$ for some compact $K \subseteq \leftopenint 0,\infty \rightopenint$. For all $r \in \leftclosedint 0,\dtwo/2 \rightopenint$, we have
\[
\esssup_{y \in \Space} |B(y,1)| \int_\Space \left| w(x,y)^r \, \Kern_{F(L)}(x,y)\right|^2 \,dx
 \leq C_{K,r} \|F\|^2_{W_2^r}.
\]
\end{prp}
\begin{proof}
Take $\chi \in C^\infty_c(\leftopenint 0,\infty \rightopenint)$ such that $\supp \chi \subseteq \leftclosedint 1/2,2 \rightclosedint$ and $\sum_{k \in \Z} \chi(2^{-k} t) = 1$ for all $t \in \leftopenint 0,\infty \rightopenint$. If $F_M$ is defined for all $M \in \leftclosedint 1,\infty \rightopenint$ as in Lemma~\ref{lem:weightedplancherel}, then
\[F(L) = \sum_{k \in \N} F_{2^k}(L,\vecT)\]
(with convergence in the strong sense). Hence an estimate for $\Kern_{F(L)}$ can be obtained, via Minkowski's inequality, by summing the corresponding estimates for $\Kern_{F_{2^k}}(L,\vecT)$ given by Lemma~\ref{lem:weightedplancherel}. On the other hand, since $|B(y,1)| \sim \max\{1,|y'|\}^\dtwo$ \cite[Proposition 3]{martini_grushin}, it is easily checked that
\[\begin{split}
\sum_{k \in \N} 2^{k(r-\dtwo/2)} \left( \chr_{\leftclosedint0,c_{K,r} \rightclosedint}(2^{-k} |y'|) + e^{-|y'|/2} \right) &\leq C_{K,r} \max\{1,|y'|\}^{r-\dtwo/2} \\
& \leq C_{K,r} \frac{(1+|y'|)^r}{|B(y,1)|^{1/2}}
\end{split}\]
when $r \in \leftclosedint 0,\dtwo/2 \rightopenint$, therefore from Lemma~\ref{lem:weightedplancherel} we obtain that
\[
|B(y,1)| \int_\Space \left| \left(\frac{|x''-y''|}{1+|y'|}\right)^r \, \Kern_{F(L)}(x,y)\right|^2 \,dx 
\leq C_{K,r} \|F\|_{W_2^r}^2.
\]
The conclusion follows by combining the last inequality with the corresponding one for $r = 0$.
\end{proof}

\section{The multiplier theorems}

Now we need some properties of the weight $w$.

\begin{lem}\label{lem:weight}
For all $x,y \in \Space$,
\[w(x,y) \leq C (1 + \dist(x,y))^{2}.\]
Moreover, if $\alpha,r \in \leftclosedint 0,\infty \rightopenint$ satisfy
\[r < \dtwo/2, \qquad \alpha + 2r > (\done+2\dtwo)/2,\]
then, for all $y \in \Space,$
\[
\int_\Space w(x,y)^{-2r} \, (1+\dist(x,y))^{-2\alpha} \,dx \leq C_{\alpha,r} |B(y,1)|.
\]
\end{lem}
\begin{proof}
Recall that $\dist(x,y) \sim \min\{\dist_1(x,y),\dist_2(x,y)\}$, where
\begin{equation}\label{eq:partialdistances}
\dist_1(x,y) = |x'-y'| + |x''-y''|^{1/2}, \qquad \dist_2(x,y) = |x'-y'| + \frac{|x''-y''|}{|x'|+|y'|},
\end{equation}
while $|B(y,1)| \sim \max\{1,|y'|\}^\dtwo$ \cite[Proposition 3]{martini_grushin}. The conclusion will then follow by proving that
\begin{gather}
\label{eq:weightdistestimatebis} w(x,y) \leq C (1+\dist_i(x,y))^2, \\
\label{eq:weightintestimatebis} \int_\Space w(x,y)^{-2r} \, (1+\dist_i(x,y))^{-2\alpha} \,dx \leq C_{\alpha,r} (1+|y'|)^\dtwo.
\end{gather}
for $i=1,2$.

As for \eqref{eq:weightdistestimatebis}, when $i=1$,
\[w(x,y) \leq (1 + |x''-y''|^{1/2})^2 \leq (1+\dist_1(x,y))^2.\]
whereas, when $i = 2$,
\[
w(x,y) = 1 + \frac{|x''-y''|}{|x'|+|y'|} \frac{|x'|+|y'|}{1+|y'|} \leq 1 + \dist_2(x,y) \, (2+|x'-y'|) \leq (1+\dist_2(x,y))^2.
\]

About \eqref{eq:weightintestimatebis}, in the case $i = 1$, since $\alpha > \done/2 + (\dtwo-2r)$, we can decompose $\alpha = \alpha' + \alpha''$ so that $\alpha' > \done/2 > 0$ and $\alpha'' > \dtwo-2r > 0$, and therefore
\begin{multline*}
\int_\Space w(x,y)^{-2r} \, (1+\dist_1(x,y))^{-2\alpha} \,dx \\
\leq \int_\Space \left(1+\frac{|x''|}{1+|y'|}\right)^{-2r} \, (1+|x'|)^{-2\alpha'} \, (1+|x''|)^{-\alpha''} \,dx \\
\leq (1+|y'|)^{2r} \int_\Space (1+|x'|)^{-2\alpha'} \, (1+|x''|)^{-2r-\alpha''} \,dx;
\end{multline*}
the last integral is finite since $2\alpha' > \done$ and $2r+\alpha'' > \dtwo$, and moreover $2r < \dtwo$.

In the case $i = 2$, instead, since $\alpha - \done/2 > \dtwo -2r$, we can choose $\alpha''$ so that $2\alpha'' \in \leftopenint \dtwo-2r, \alpha - \done/2\rightopenint$; in particular $0 < \alpha'' < \alpha/2$, hence $\alpha' = \alpha - \alpha'' > \alpha/2 > 0$. Then
\begin{multline*}
\int_\Space w(x,y)^{-2r} \, (1+\dist_2(x,y))^{-2\alpha} \,dx \\
\leq C_{\alpha,r} \int_\Space  
\left(1+\frac{|x''|}{1+|y'|}\right)^{-2r} \, (1+|x'|)^{-2\alpha'} \, \left(1+\frac{|x''|}{1+|x'|+|y'|}\right)^{-2\alpha''} \,dx \\
\leq C_{\alpha,r} \int_\Space \left(1+\frac{|x'|}{1+|y'|}\right)^{2\alpha''} \left(1+\frac{|x''|}{1+|y'|}\right)^{-2r-2\alpha''} \, (1+|x'|)^{-2\alpha'} \,dx.
\end{multline*}
Since $2\alpha'' + 2r > \dtwo$, the integral in $x''$ converges, and moreover $2\alpha'' > 0$, hence the denominator $1+|y'|$ in the first factor can be discarded, and we obtain
\begin{multline*}
\int_\Space w(x,y)^{-2r} \, (1+\dist_2(x,y))^{-2\alpha} \,dx \leq C_{\alpha,r} (1+|y'|)^\dtwo \int_{\R^\done} (1+|x'|)^{-2\alpha'+2\alpha''} \,dx'.
\end{multline*}
Since $2\alpha'-2\alpha'' = 2(\alpha  - 2\alpha'') > \done$, the integral in $x'$ converges too, and we are done.
\end{proof}

Via interpolation, we are now able to give a strengthened version of the standard weighted $L^2$ estimate due to the Gaussian heat kernel bounds for $L$ (see \cite[Proposition 11]{martini_grushin} and references therein).

\begin{prp}\label{prp:interpolatedestimate}
Let $\alpha,\beta,r \in \leftclosedint 0,\infty \rightopenint$ be such that $r < \dtwo/2$ and $\beta > \alpha + r$. Let $K \subseteq \leftopenint 0,\infty \rightopenint$ be compact. For all smooth $F : \R \to \C$ with $\supp F \subseteq K$, we have
\[
\esssup_{y \in \Space} |B(y,1)|^{1/2} \left\| w(\cdot,y)^r \, (1+\dist(\cdot,y))^\alpha \, \Kern_{F(L)}(\cdot,y)\right\|_2 \,dx
 \leq C_{K,r,\alpha,\beta} \|F\|_{W_2^\beta}.
\]
\end{prp}
\begin{proof}
For $\alpha = 0$ and $\beta \geq r$, the inequality is given by Proposition~\ref{prp:weightedplancherel}.

On the other hand, for arbitrary $\alpha$, if $\beta > \alpha + 2r + 1/2$, then the inequality follows from Lemma~\ref{lem:weight} and \cite[Proposition 11]{martini_grushin}.

The full range $\beta > \alpha + r$ is then recovered by interpolation (cf.\ \cite[Lemma~1.2]{mauceri_vectorvalued_1990} and \cite[Proposition 13]{martini_grushin}).
\end{proof}

We are finally able to prove the fundamental estimate, and consequently our theorems.

\begin{proof}[Proof of Proposition~\ref{prp:l1estimate}]
Since the operator $L$ and the distance $\dist$ are homogeneous with respect to the dilations $\dil_r$, it is not restrictive to assume that $R = 1$.

Let $r,\alpha' \in \leftclosedint 0,\infty \rightopenint$. For all $y \in \Space$, H\"older's inequality gives
\begin{multline*}
\| (1+\dist(\cdot,y))^\alpha \, \Kern_{F(L)}(\cdot,y)\|_1 \leq \left(\int_\Space w(x,y)^{-2r} \, (1+\dist(x,y))^{-2\alpha'} \,dx \right)^{1/2} \\
\times \| w(\cdot,y)^r \, (1+\dist(\cdot,y))^{\alpha+\alpha'} \, \Kern_{F(L)}(\cdot,y)\|_2.
\end{multline*}
The first factor on the right-hand side can be controlled by Lemma~\ref{lem:weight} if
\[r < \dtwo/2, \qquad \alpha' + 2r > (\done+2\dtwo)/2,\]
while the second factor can be controlled by Proposition~\ref{prp:interpolatedestimate} if moreover
\[\beta > \alpha+\alpha'+r.\]
Under our hypotheses, $\epsilon \defeq \beta - \alpha - (\done+\dtwo)/2 > 0$; therefore, if we choose $r \in \leftopenint d_2/2-\epsilon,d_2/2\rightopenint$ and $\alpha' \in \leftopenint d_1/2+d_2-2r,\beta-\alpha-r\rightopenint$, then the above conditions are satisfied, and we are done.
\end{proof}

\bibliographystyle{amsabbrv}
\bibliography{grushin2}

\end{document}